\documentclass[11pt]{amsart}
\usepackage{amsmath,amssymb, amscd,amsthm, amsfonts}
\usepackage[mathscr]{eucal}
\usepackage[all]{xy}
\usepackage[margin=2.5cm]{geometry}

\newtheorem{theorem}{Theorem}[section]
\newtheorem{lemma}[theorem]{Lemma}
\theoremstyle{definition}
\newtheorem{definition}[theorem]{Definition}

\newtheorem{proposition}[theorem]{Proposition}
\theoremstyle{remark}
\newtheorem{remark}[theorem]{Remark}
\numberwithin{equation}{section} 
\theoremstyle{plain}
\newtheorem*{acknowledgement}{Acknowledgement}

\newtheorem{corollary}[theorem]{Corollary}

\def\r{\mathcal R}
\def\s{\mathbb S}
\def\C{\mathbb C}
\def \nequiv{\not \equiv}
\def\B{\mathbb B}

\def\H{\mathbb H}

\def\V{\mathbb V}
\def\W{\mathbb W}

\def\Z{\mathbb Z}
\def\E{\mathbb E}

\def\R{\mathbb R}

\def\V{\mathbb V}
\def\e{\mathcal E}
\def\h{\mathcal H}
\def\p{\mathcal P}
\def\d{\mathcal D}
\def\g{\mathcal G}
\def \o{\mathcal O}
\def \a{\mathcal A}
\def\t{\mathcal T}
\def\P{\mathbb P}

\newcommand{\secref}[1]{section~\ref{#1}}
\newcommand{\thmref}[1]{Theorem~\ref{#1}}
\newcommand{\lemref}[1]{Lemma~\ref{#1}}
\newcommand{\remref}[1]{Remark~\ref{#1}}
\newcommand{\propref}[1]{Proposition~\ref{#1}}
\newcommand{\corref}[1]{Corollary~\ref{#1}}

\begin{document}
\title[Conjugacy classes in M\"obius groups]
{Conjugacy classes in M\"obius groups  }
\author[Krishnendu Gongopadhyay]{Krishnendu Gongopadhyay}
\address{Indian Institute of Science Education and Research (IISER) Mohali, Transit Campus: MGSIPAP Complex, Sector 26 Chandigarh 160 019, INDIA}
\email{krishnendug@gmail.com}
\begin{abstract}
Let $\H^{n+1}$ denote the $n + 1$-dimensional (real) hyperbolic space. Let $\s^{n}$
denote the conformal boundary of the hyperbolic space. The group of conformal diffeomorphisms of $\s^n$ is denoted by $M (n)$. Let $M_o (n)$ be its identity component which consists of all orientation-preserving elements in $M (n)$. The conjugacy classification of isometries in $M_o (n)$ depends on the conjugacy of $T$ and $T^{ -1}$ in $M_o (n)$. For an element $T$ in $M (n)$, $T$ and $T^{ -1}$ are conjugate in $M (n)$, but they may not be conjugate in $M_o (n)$. In the literature, $T$ is called real if $T$ is conjugate in $M_o (n)$ to $T^{ -1}$ . In this paper we classify real elements in $M_o (n)$. 

Let $T$ be an element in $M_o(n)$. Corresponding to $T$ there is an associated element $T_o$ in $SO(n+1)$. If the complex conjugate eigenvalues of $T_o$ are given by $\{e^{i\theta_j}, e^{-i\theta_j}\}$, $0 < \theta_j \leq \pi$, $j=1,...,k$, then $\{\theta_1,...,\theta_k\}$ are called the \emph{rotation angles} of $T$.  If the rotation angles of $T$ are distinct from each-other, then $T$ is called a \emph{regular} element. After classifying the real elements in $M_o (n)$  we have parametrized the conjugacy classes of regular elements in $M_o (n)$. In the parametrization, when $T$ is not conjugate to $T^{ -1}$ , we have enlarged the group and have considered the conjugacy class of $T$ in $M (n)$. We prove that each such conjugacy class can be induced with a fibration structure.
\end{abstract}
\maketitle

\section{Introduction}
Let $\H^{n+1}$ denote the $n+1$-dimensional (real) hyperbolic space. Let $\s^{n}$ denote the conformal boundary of the hyperbolic space. The group of conformal diffeomorphisms of $\s^{n}$ is denoted by $M(n)$. Let $M_o(n)$ be its identity component which consists of all orientation-preserving elements in $M(n)$.  By Poincar\'e extensions, the group of isometries of $\H^{n+1}$ is identified with the group $M(n)$. Throughout the article, in the ball model and the upper-half space model of the hyperbolic space, we denote the isometry group and its identity component by $M(n)$ and $M_o(n)$ respectively. The underlying model of $\H^{n+1}$ would be clear from the context. In this paper we aim to classify the conjugacy classes in $M_o(n)$ and to parametrize them. 

The conjugacy classes in the group $M(n)$, $n \geq 1$,  have been well understood due to the work of several authors, cf. Chen-Greenberg \cite{cg}, Gongopadhyay-Kulkarni \cite{kg},  Greenberg \cite{greenberg}, Kulkarni \cite{kulkarni}, also see Ahlfors \cite{ahlfors}, Cao-Waterman \cite{cw} and Wada \cite{wada} for a viewpoint using Clifford numbers. Using the linear (hyperboloid) model of the hyperbolic space, Gongopadhyay-Kulkarni \cite[Theorem 1.2]{kg}  proved that the conjugacy classes in $M(n)$ are determined by the minimal and the characteristic polynomial of an isometry. In particular, an element in $M(n)$ is conjugate to its inverse. However, this is not true for elements in $M_o(n)$. For example, an unipotent isometry $T$ in $M_o(1)$ is not conjugate to its inverse. This can be seen easily by identifying $M_o(1)$ with $PSL(2, \R)$. Thus a conjugacy class in $M(n)$ may breaks into conjugacy classes in $M_o(n)$.  This happen especially if an isometry $T$ is not conjugate to its inverse in $M_o(n)$. Thus to determine the conjugacy classes in $M_o(n)$ it is important to determine the elements which are conjugate in $M_o(n)$ to its own inverse. The classification of such elements essentially gives the conjugacy classification in $M_o(n)$. 

\subsubsection*{The Real elements in M\"obius groups}Following Feit-Zuckerman \cite{fz}, an element $g$ in a linear algebraic group $G$ 
is said to be \emph{real} if it is conjugate in $G$ to its own inverse. Thus every element in $M(n)$  is real. Reality properties of elements in linear algebraic groups have been studied by many authors due to their role in representation theory 
cf. Feit-Zuckerman  \cite{fz}, Moeglin et. al. \cite{mvw}, Singh-Thakur \cite{st1, st2}, 
Tiep-Zalesski \cite{tz}. The reality properties of the elements in $SO(n,1)$ follow from recent results of Singh-Thakur \cite[Theorem 3.4.6]{st1} combining it with earlier results of Kn\"uppel-Nielsen \cite{kn} and Wonenburger \cite{w}. However, the results of these authors do not carry over to the identity component $SO_o(n,1)$, and none of these authors have addressed this issue either. In this paper we have given a classification of the real elements in the group $SO_o(n,1)$. As an immediate corollary we obtain the classification of real elements in $M_o(n)$.

We work with the linear model (or the hyperboloid model) of the hyperbolic space to determine the reality properties of the conjugacy classes in $M_o(n)$. Let $\V$ be a real vector space of dimension $n+1$ equipped with a non-degenerate quadratic form $Q$ of signature $(n,1)$, i.e. with respect to a suitable coordinate system $Q$ has the form {$Q(x)=x_0^2 +...+x_{n-1}^2-x_{n}^2$}.  Let $O(n,1)$ denote the full group of isometries of $(\V,Q)$. Let $SO(n,1)$ be the index $2$ subgroup of all isometries with determinant $1$. It has two components. We call  $v \in \V$ {\it time-like} (resp. {\it space-like}, resp. {\it light-like}) if
$Q(v) < 0$, (resp. $Q(v) >  0$, resp. $Q(v) = 0$). A subspace $\W$ is  {\it time-like} (resp. {\it space-like}, resp. {\it light-like}) if $Q|_{\W}$ is non-degenerate and   indefinite (resp. $Q|_{\W} > 0$, resp. $Q|_{\W} = 0$). The hyperboloid
$\{v\in \V\;|\; Q(v) = -1\}$ has two components. The component containing the vector 
$e_n=(0,0,...,0,1)$ is taken to be the \emph{ linear or hyperboloid model} of the hyperbolic space $\H^n$. The isometry group $I(\H^n)$ is the index $2$ subgroup of $O(n,1)$ which preserves $\H^{n}$.   The group  $M(n)$ is identified with $I(\H^{n+1})$, and with this identification $M_o(n)=SO_o(n+1,1)$.   Though the groups $M(n)$ and $SO(n+1,1)$ have the same identity component, they differ from each-other due to their second components. The second component in $M(n)$ consists of the orientation-reversing isometries of $\H^{n+1}$ which are of determinant $-1$.

    \medskip  Our main theorem concerning the reality properties of the conjugacy classes in $SO_o(n,1)$ 
 is the following. 
\begin{theorem}\label{realson1}
1. Every element in $SO_o(n,1)$ is real if and only if $n \equiv 0\;(\hbox{mod }4)$ or $n \equiv 3 \;(\hbox{mod }4)$. 

2. If $n \equiv 1 \;(\hbox{mod }4)$, then an element $T$ in $SO_o(n,1)$ is real if and only if it is either a hyperbolic isometry with at least one eigenvalue $1$ or $-1$,  or,  it is not hyperbolic.

3.  If $n \equiv 2 \;(\hbox{mod }4)$, then an element $T$ in $SO_o(n,1)$ is real if and only if 
one of the following holds.

(i) $T$ is hyperbolic, (ii) $T$ is a non-hyperbolic with at least one eigenvalue $-1$, (iii) $T$ is non-hyperbolic, it has no eigenvalue $-1$, and there is at least one eigenvector to $1$ which is space-like. 
\end{theorem}
The theorem is proved in \secref{real}. This theorem, along with \cite[Theorem 1.2]{kg},  provide a better understanding of the conjugacy classes in $M_o(n)$. As an immediate corollary we have the following. 

\begin{corollary}\label{mn}
Every element in the group $M_o(n)$ is real if and only if  
$n \equiv 2\;(\hbox{mod }4)$ or $n \equiv 3 \;(\hbox{mod }4)$. 
\end{corollary}
An element $T$ in $SO_o(n,1)$ is called \emph{strongly real} if it is a product of two involutions in $SO_o(n,1)$.  
\begin{corollary}\label{1}
An element $T$ in $SO_o(n,1)$ is strongly real if and only if it is real. 
\end{corollary}
Much after the completion of this work,  I have come to know about the work of Short et. al. \cite{sh1, sh2}. What we have called \emph{real}, resp. \emph{strongly real elements}, these authors have called them \emph{reversible}, resp. \emph{strongly reversible} in their papers. Motivated by the viewpoint of physics and dynamics \cite{bir, d, mw}, these authors are also interested in  studying the reality properties of conjugacy classes in groups. The three body problem in Birkhoff \cite{bir} is reversible in the sense that one can reverse time and the problem stays much the same. Devaney \cite{d} compared reversible and Hamiltonian reversible dynamical systems. The notion of strongly reversible elements was implicit in Moser-Webster \cite{mw}. These background materials motivated  Short et. al. to introduce the terminology `reversible' and  `strongly reversible'.  They have, however,  mostly overlooked the literature on real elements, which has grew out from the algebraic viewpoint.  

Using the ball model, Short \cite{sh1} has  studied the reality properties of conjugacy classes in the M\"obius groups.  The  approach of Short is geometric.  In particular, Short has a geometric proof of \corref{mn}. In the other cases, the classification of Short are different from that of us.

\subsubsection*{The Topology of the Conjugacy classes} After the classification of the conjugacy classes, we aim to parametrize them. For this purpose we appeal to the ball model and the upper-half space model of the hyperbolic space.  We call two isometries are in the same \emph{fixed-point class} if their fixed-point classifications are the same, i.e. both are either elliptic, or  parabolic,  or hyperbolic. 
The \emph{stretch factor} of a hyperbolic isometry corresponds to its real eigenvalues which are different from $1, \  -1$,  in the linear model. The stretch factors of elliptics and parabolics are defined to be one. The stretch factors of two hyperbolic isometries are a \emph{unit} if they are inverse to each-other. 

 If an element $S$ in $SO(n+1)$ has $k$ rotation angles, then it is called a \emph{$k$-rotation}. If the rotation angles of $S$ are distinct from each-other, then $S$ is called a \emph{regular $k$-rotation}. 

Let $T$ be an element in $M_o(n)$. Corresponding to $T$ there is an associated element $T_o$ in $SO(n+1)$. If the complex conjugate eigenvalues of $T_o$ are given by $\{e^{i\theta_j}, e^{-i\theta_j}\}$, $0 < \theta_j \leq \pi$, $j=1,...,k$, then $\{\theta_1,...,\theta_k\}$ are called the \emph{rotation angles} of $T$.  The element $T$ is called a \emph{$k$-rotatory elliptic}, resp. \emph{parabolic}, resp. \emph{ hyperbolic} if $T$ is elliptic, resp. parabolic, resp. hyperbolic, and $T_o$ is a $k$-rotation. If $T_o$ is a regular $k$-rotation, then $T$ is called a \emph{regular} isometry.

 Classically a rotation angle or an \emph{angle} of an isometry was defined to be an element in $[-\pi, \pi]$, cf. Greenberg \cite{greenberg}. We restrict them in the interval $(0, \pi]$ in order to  associate them to the complex conjugate eigenvalues which are conjugacy invariants. Thus according to our definition $T$ and $T^{-1}$ have the same rotation angles. The details of this classification is given in \secref{class}.

Let $\mathfrak S(2;k)$ denote the space of all decompositions of $\E^{2k}$ into orthogonal sum of two dimensional subspaces. Let $\oplus$ denote the orthogonal sum. Then 
$$\mathfrak S(2; \ k)=\{(\V_1,....,\V_k)\;|\;\hbox{for each }i,\hbox{ dim }\V_i=2, \hbox{ and} \;\E^{2k}=\oplus_{i=1}^k\V_i\}.$$
For $k=0$, we define $\mathfrak S(2; \ 0)$ to be a single point. The space $\mathfrak S(2; \ k)$ is homeomorphic to $O(2k)/O(2)^k$. 
Two decompositions $D_1$ and $D_2$ in $\mathfrak S(2; \ k)$ are said to be equivalent if $D_1$ is obtained from $D_2$ by a permutation of the summands. That is, if $D_1: \E^{2k}=\oplus_{i=1}^k \V_i$ is a given decomposition, then an element in its equivalence class is given by $\E^{2k}=\oplus_{j=1}^k \V_{i_j}$, where $(i_1,...,i_k)$ is a permutation of $(1,2,...,k)$. The set of equivalence classes of the decompositions is denoted by $\d(2; \ k)$. Let $S_k$ denote the symmetric group of $k$ symbols. Then $S_k$ acts freely on $\mathfrak S(2; \ k)$. Thus $\d(2; \ k)$ is the quotient space   $\mathfrak S(2; \ k)/S_k$. 

Let $\Delta=\{(x,x)\;|\;x \in \s^n\}$. Then $\Delta$ is a closed subspace of $\s^n \times \s^n$. The space $\tilde \Delta_n$ is the open subspace $\s^n \times \s^n-\Delta$. The space $\B_n$ is the quotient space of $\tilde \Delta_n$ obtained from the equivalence relation $(x,y) \sim (y,x)$. Let $\g_k(n)$ denote the Grassmannian of $k$-dimensional vector subspaces of $\R^n$ and $\mathcal S_k(n)$ denote the Grassmannian of $k$-dimensional spheres in $\s^n$. 

To avoid confusions regarding reality issue, we further adopt the convention of taking
the conjugacy class of $T$ in $M (n)$. With this convention, two elements in $M_o (n)$ are conjugate if and only if they are in the same fixed-point class, have the same set of rotation angles, and their stretch factors are either the same or a unit.  Some parametrizations of the conjugacy classes for $M_o(1)$ is already available in the literature, cf. Falbel-Wentworth \cite[p-6]{fw} and Kulkarni-Raymond 
\cite[ p-242]{kr}, also
see Kulkarni \cite[p-52]{kulkarni}. In higher dimensions, it is not expected to have similar parametrization of the conjugacy classes. However,  the conjugacy classes of regular isometries carry a natural parametrization  even in higher dimensions. 
\begin{theorem}\label{tcco}
 Let $T$ be a regular $k$-rotation of $\E^n$ with rotation angles $\Theta=\{\theta_1,...,\theta_k\}$, $0 < \theta_i \leq \pi$., i.e.  for all $i, \ j$, $\theta_i \neq \theta_j$. Let $\o_{k, \Theta}$ denote the conjugacy class of $T$ in $O(n)$. 

(i) When $n = 2k$, the conjugacy class $\o_{k, \Theta}(2k)$ is a $d_0(k)$ sheeted covering space over $\d(2; \  k)$, where 
$$d_0(k)= \left \{ \begin{array}{ll} 2^k k! \ \hbox{ if $T$ has no rotation angle $\pi$,}\\
                    2^{k-1}k! \ \hbox{ otherwise.}
                   \end{array}  \right.$$ 

(ii) When $n > 2k$, the conjugacy class $\o_{k, \Theta}(n)$ is the total space of a fibration with base $\g_{n-2k} (n)$ and fiber $\o_{k, \Theta} (2k)$. 
\end{theorem}
It follows from the above theorem that  the induced topological structure of the conjugacy class of a regular $k$-rotation depends on $k$ and whether $\pi$ is a rotation angle or not.  It is independent of the actual numerical values of the rotation angles which are different from $\pi$.  We identify the conjugacy classes of the regular $k$-rotations with  no rotation angle $\pi$, and denote it by $\o_k(n)$. The conjugacy classes of the regular $k$-rotations with a rotation angle $\pi$ are also identified, and we denote it by $\o_k^-(n)$. 

\begin{theorem}\label{tcc}
Let $T$ be a regular isometry of $\H^{n+1}$ with rotation angles  $\Theta=\{\theta_1,...,\theta_k\}$. Assume that all the rotation angles of $T$ are different from $\pi$. 

(i) Let $T$ be a  $k$-rotatory elliptic. For $k \leq [ \frac{n}{2}]$, the conjugacy class of $T$  in $M(n)$ is the total space of a fibration with base $\mathcal S_{n-2k} (n)$ and fiber $\o_{k} (2k)$.

   Suppose $n$ is an odd number. Then the conjugacy class of $T$  is the total space of a fibration with base $\H^{n+1}$ and fiber $\o_{\frac{n+1}{2}} (n + 1)$.

(ii) Let $T$ be a $k$-rotatory parabolic. The conjugacy class of  $T$ in $M(n)$ is the total space of a fibration with base $\s^n$ and fiber $\o_{k} (n) \times {\R^n}^{\ast}$, where ${\R^n}^{\ast}$ denote the punctured $n$-dimensional affine space.
  
(iii) Let $T$ be a  $k$-rotatory hyperbolic.   The  conjugacy class of  $T$  in $M(n)$  is the total space of a fibration with base $\B_n$ and fiber $\o_{k}(n) \sqcup \o_k(n)$, where $\sqcup$ denote the disjoint union.
\end{theorem}
When $T$ is a regular isometry with a rotation angle $\pi$, then the same theorem holds true with the understanding that for each $p, q$, $\o_p(q)$ is replaced by $\o^-_p(q)$ in the statement. 

These theorems are established in \secref{cco} and \secref{cc} respectively.  The theorems are established by identifying the conjugacy classes with some canonically chosen topological fibrations. These fibrations are obtained from the conjugacy invariants, and the eigenspace decomposition of the  corresponding $k$-rotation. To each regular $k$-rotation, we can assign a unique orthogonal decomposition into two dimensional subspaces corresponding to complex conjugate eigenvalues. This is no more true for non-regular $k$-rotations, and roughly, this is the reason that there is no canonical topological structure on the conjugacy classes of non-regular isometries, cf. \remref{2}. 

\section{Preliminaries}\label{prel}
\subsection{The space $\tilde \Delta_n$ and $\B_n$}
\ Let $\Delta=\{(x,x)\;|\;x \in \s^n\}$. Then $\Delta$ is a closed subspace of $\s^n \times \s^n$. The space $\tilde \Delta_n$ is the open subspace $\s^n \times \s^n-\Delta$. The space $\B_n$ is the quotient space of $\tilde \Delta_n$ obtained from the equivalence relation $(x,y) \sim (y,x)$. 

The space $\tilde \Delta_n$ is  path-connected. Note that for $n \geq 2$ on the sphere $\s^n$,  given three points there exists a path joining two points, and avoiding the third point. Let $(x_1, y_1)$ and $(x_2, y_2)$ be two distinct points on $\tilde \Delta_n$. On the copy $\s^n \times \{y_1\}$, there is a path $\gamma_1$ joining $(x_1, y_1)$ and $(x_2, y_1)$, and not passing through $(y_1, y_1)$. Similarly there is a path $\gamma_2$ joining $(x_2, y_1)$ and $(x_2, y_2)$ on $\{x_2\} \times \s^n$ such that $(x_2, x_2)$ does not belong to $\gamma_2$. Then the composite $\gamma_1 \circ \gamma_2$ is a path on $\tilde \Delta_n$ joining $(x_1, y_1)$ and $(x_2, y_2)$. This shows that the space $\tilde \Delta_n$ is path-connected. Hence for $n \geq 2$, $\tilde \Delta_n$ is connected. It is easy to see that $\tilde \Delta_1$ is also connected. The relation $\sim$ defines a free $\Z_2$ action on $\tilde \Delta_n$. Hence the projection map $p: \tilde \Delta_n \to \B_n$ is a covering. 

\subsection{The Grassmannians}\label{gr}
The Grassmannian $\g_k(n)$ is the collection of all $k$-dimensional vector subspaces of $\R^n$, and $\a_k(n)$ is the collection of all $k$-dimensional affine subspaces of $\R^n$. The Grassmannians are well known topological manifolds. The space $\g_k(n)$ has dimension $k(n-k)$, and the space $\a_k(n)$ has dimension $(k+1)(n-k)$. When $k=1$, the space $\g_1(n)$ is the real projective space $\P \R^n$.

Now consider the collection $\mathcal S_k(n)$ of  $k$-dimensional spheres in $\s^n$. 
A $k$-sphere in $\s^n$ is a section of $\s^n$ by an affine $k+1$-dimensional subspace of $\R^{n+1}$. This identifies $\mathcal S_k(n)$  with an open subspace of $\a_{k+1}(n+1)$, and thus has dimension $(k+2)(n-k)$. As a homogeneous space $\mathcal S_k(n)$ is homeomorphic to $M(n)/(M(k) \times O(n-k))$, where $M(n)$ denote the full group of conformal diffeomorphisms of $\s^n$.  

The $0$-dimensional spheres on $\s^n$ are defined to be 
the subsets containing two distinct points. Hence 
$$\mathcal S_o(n)=\{ \{x, y\}\;| \; x, y \in \s^n, \; x \neq y\}.$$ 
We identify $\mathcal S_o(n)$ with the space $\B_n$. 

 Let $\mathfrak S(2;k)$ denote the space of all decompositions of $\E^{2k}$ into orthogonal sum of two dimensional subspaces. Let $\oplus$ denote the orthogonal sum. Then 
$$\mathfrak S(2;k)=\{(\V_1,....,\V_k)\;|\;\hbox{for each }i,\hbox{ dim }\V_i=2, \hbox{ and} \;\E^{2k}=\oplus_{i=1}^k\V_i\}.$$
For $k=0$, we define $\mathfrak S(2; \ 0)$ to be a single point. 

{\lemma The set $\mathfrak S (2; \ k)$ can be given a natural topology such that it has  dimension $2k(k-1)$.}
\begin{proof}
The group $O(2k)$ acts transitively
on $\mathfrak S(2; \ k)$. The stabilizer subgroup at a point is isomorphic
to the cartesian product of $k$ copies of $O(2)$. Thus the set $\mathfrak S(2; \ k)$ has a bijection with the the space $O(2k)/O(2)^k$. Using this bijection we induce the topology on $\mathfrak S(2; \ k)$ such that it is homeomorphic to $O(2k)/O(2)^k$. 
Thus the dimension of $\mathfrak S(2; \ k)$ is 

$\hbox{dimension }O(2k)- \hbox{dimension } 
O(2)^k=k(2k-1)-k=2k(k-1).$
\end{proof}

\section{Reality properties of conjugacy classes in $M_o(n)$: Proof of \thmref{realson1}}\label{real} 
Let $G$ be a group. An element $g$ in $G$ is called \emph{real} if there exists $h$ in $G$ such that $hgh^{-1}=g^{-1}$. An element $g$ in $G$ is an \emph{involution} if $g^2=1$. An element $g$ in $G$ is called \emph{strongly real} if it is a product of two involutions in $G$. Note that a strongly real element is always real in $G$. Conversely, a real element $g$ in $G$ is strongly real if and only if there is a conjugating element in $G$ which is an involution. 

\subsection{Reality in $SO(n)$}
Let $\V$ be an $n$-dimensional vector space over $\R$ equipped with a non-degenerate positive definite quadratic form $q$. Let $O(n)$ denote the isometry group and let $SO(n)$ denote the index two subgroup of $O(n)$ consisting of all isometries with determinant $1$. We identify $\V$ with $\E^n$. 
The following theorem follows from a result of Wonenburger \cite{w}.
\begin{theorem}\label{w}{\bf[W]}
 Every element in the orthogonal group is strongly real. Hence every element in $O(n)$ is real. 
\end{theorem}

The following theorem by Kn\"uppel and Nielsen  ( \cite{kn} Theorem B)  classifies the strongly real elements in $SO(n)$. 

\begin{theorem}\label{kn} {\bf [KN]}  Let $T$ be an element in $SO(n)$. Then $T$ is strongly real in $SO(n)$ if and only if $n \not \equiv 2\;(\hbox{mod }4)$ or an orthogonal decomposition of $\V$ into orthogonally indecomposable $T$-invariant subspaces contains an odd dimensional summand. 
\end{theorem}
\begin{proof}
Let $T$ be in $SO(n)$ and has no eigenvalue $1$, $-1$. Then $\V$ has an orthogonal decomposition 
$$\V=\V_1 \oplus .... \oplus \V_k,$$
into two dimensional invariant subspaces.  By the two dimensional case, for each $i=1,...,k$, there is an involution $f_i$ such that $f_i T|_{\V_i} f_i^{-1}=T|_{\V_i}^{-1}$. Let $f=f_1 \oplus f_2 \oplus ... \oplus f_k$. Then $f$ is an involution, and $fTf^{-1}=T^{-1}$, and $\det f=(-1)^{\frac{n}{2}}$. Thus $\det f$ can be made $1$ if and only if either of the conditions stated in the theorem is satisfied.  If $\det f=1$, then $T=f.fT$, a product of two involutions. 

This completes the proof.  
\end{proof}
\begin{lemma}\label{lem1}
Suppose $n$ is even and $T$ be an element in $SO(n)$. Suppose the minimal polynomial of $T$ is a power of an irreducible quadratic polynomial over $\R$. Then $T$ is real if and only if $n \not \equiv 2 \  (\hbox{mod }4)$.  
\end{lemma}

\begin{proof}
Since the minimal polynomial of $T$ is a irreducible power polynomial over $\R$, let its only eigenvalues over $\C$ are $\{e^{i \theta}, e^{-i \theta}\}$. Let the characteristic polynomial of $T$ be 
\\ $\chi_T(x)=(x^2 -2 \cos \theta \ x +1)^m$. Then $n=2m$. Let $S$ be an element in $O(n)$ such that $STS^{-1}=T^{-1}$.

Let $\V_c = \V \otimes_\R \C$  be its complexification. We identify $T$ with $T\otimes_{\R} id$ and also consider it as an operator on $\V_c$. Let $\V_c =\V_{\theta} + \V_{-\theta}$ be the decomposition into its eigenspaces. Then $S$ interchanges $\V_{\theta}$ and $\V_{-\theta}$. It follows that $\det S=(-1)^m$. Hence $S$ is an element of $SO(n)$ if and only if $m$ is even.  
\end{proof}

\begin{proposition}\label{soprop}
Let $T$ be an element in $SO(n)$. Suppose $1$ and $-1$ are not eigenvalues of $T$.  Then $T$ is real if and only if $n \not \equiv 2\; (\hbox{mod }4)$.  
\end{proposition}

\begin{proof}
If $1$ or $-1$ is not an eigenvalue of $T$, then $n$ must be even. For $T$ in $SO(n)$ there exists a decomposition of $\V$ into $T$-invariant  subspaces 
$$\V =\V_1 \oplus \V_2 \oplus ... \oplus \V_k,$$
where  for each $i=1,2,...,k$, $\V_i\simeq \R[x]/(x^2-2 \cos \theta_i \ x +1)^{m_i}$ for $m_i \geq 1$. Let $S$ be an element in $O(n)$ such that $STS^{-1}=T^{-1}$. Then $S$ keeps each $\V_i$ invariant. Let $S_i=S|_{\V_i}$, $T_i=T|_{\V_i}$. Then $S_i T_i S_i^{-1}=T_i^{-1}$. It follows from the proof of \lemref{lem1} that $\det S_i=(-1)^{m_i}$. 
Thus $\det S=\Pi_{i=1}^k \det S_i=(-1)^{\frac{n}{2}}$. Hence $\det S=1$ if and only if $\frac{n}{2}=2m$. This proves the theorem.  
\end{proof}

The following theorem characterizes reality in $SO(n)$.  It may be considered as a special case of the results of Kn\"uppel-Nielsen \cite{kn} and Singh-Thakur \cite{st1}. The arguments we have used in the proof is motivated by Singh-Thakur \cite[section 3.4]{st1}.

\begin{theorem}\label{thmson}
Let $T$ be an element in $SO(n)$. Then $T$ is real in $SO(n)$ if and only if either $n \not \equiv 2 \;(\hbox{mod }4)$ or $T$ has an eigenvalue $1$ or $-1$. 
\end{theorem}

\begin{proof}
Suppose $S$ in $O(n)$ be such that $STS^{-1}=T^{-1}$. If $T$ has no eigenvalue $\pm 1$ and $n \equiv 2\;(\hbox{mod }4)$, then it follows from the above proposition that $S$ can not have determinant $1$. Hence $T$ is not real. 

Suppose $T$ has an eigenvalue $1$ or $-1$.  If $-1$ is an eigenvalue of $T$,  it must have an even multiplicity.  Thus we have an orthogonal decomposition of $\V$ into $T$-invariant subspaces
$$\V=\V_1 \oplus .... \oplus \V_k \oplus \Lambda_1 \oplus \Lambda_{-1},$$
where each $\V_i$ is $T$-invariant and even dimensional, $T|_{\Lambda_1}=I$,  $T|_{\Lambda_{-1}}=-I$. Suppose $\W=\V_1 \oplus .... \oplus \V_k$. Then dimension of $\W$ is even.  By the reality of the orthogonal group, there exists an orthogonal map  $S_w: \W \to \W$ such that $S_wT_wS_w^{-1}=T_w^{-1}$ and $\det S_w = I$ or $-1$. 
Since the maps $I$ and $-I$ commutes with any element in the orthogonal group, after choosing such $S_w$, the maps $S_1:{\Lambda} \to \Lambda$ or $S_{-1}:\Lambda^{-1} \to \Lambda^{-1}$ can be chosen 
accordingly such that $STS^{-1}=T^{-1}$ and $\det S=1$, where $S=S_w \oplus S_1 \oplus S_{-1}$.  Hence $T$ is real in $SO(n)$. 

This completes the proof. 
\end{proof}
\begin{corollary}
An element $T$ in $SO(n)$ is real if and only if it is strongly real. 
\end{corollary}
\begin{proof}
The corollary follows from the above theorem, combining it with \thmref{kn}.[KN]. 
\end{proof}

\subsection{Proof of \thmref{realson1}} 

Suppose $T$ is an element in $SO_o(n,1)$. Let $T$ be elliptic. Then $T$ fixes a time-like eigenvector $v$. Let $\W$ be the space-like orthogonal complement to the one-dimensional subspace spanned by $v$. Then $T_o=T|_{\W}$ is an element in $SO(n)$. If $n \nequiv 2\;(\hbox{mod }4)$, then there exists an orthogonal map $S_o: \W \to \W$ such that $\det S_o=1$ and $S_oT_oS_o^{-1}=T_o^{-1}$. Hence there exists $S=\begin{pmatrix}S_o & 0 \\0 & 1\end{pmatrix}$ in $SO_o(n,1)$ such that $STS^{-1}=T^{-1}$. Thus $T$ is real in $SO_o(n,1)$. 

Suppose $n \equiv 2\;(\hbox{mod }4)$, and suppose $T$ has no space-like eigenvalue $1$, $-1$. In this case, as in the proof of \propref{soprop},  any choice of $S_o$ has determinant necessarily $-1$. Thus it is not possible to choose any $S$ as above. Hence $T$ can not be real. 

Suppose $T$ is hyperbolic. Then $T$ has a real eigenvalue $r>0$. Consequently, $\V$ has an orthogonal decomposition $\V=\V_r \oplus \W$, where $\V_r$ is a $2$-dimensional orthogonally indecomposable time-like subspace and $\W$  is its space-like orthogonal complement of dimension $(n-1)$. Denote $T_r=T|_{\V_r}$, $T_o=T|_{\W}$. Since $T$ is semisimple, and $\V_r$ is an eigenspace of $T$, it follows that any element $S$ which conjugates $T$ to $T^{-1}$ must preserve $\V_r$. It is easy to see that any $f$ in $I(\H^1)$ such that $fT_rf^{-1}=T_r^{-1}$ must have determinant $-1$. Thus $T$ is real in  $SO_o(n,1)$  if and only if we can choose an $S_o$ in $O(n-1)$ such that $S_oT_oS_o^{-1}=T_o^{-1}$ and $\det S_o=-1$. This is the case precisely when $n-1 \nequiv 0\;(\hbox{mod }4)$, i.e. 
$n \nequiv 1\;(\hbox{mod }4)$, or $T_o$ has an eigenvalue $\pm 1$.

Suppose $n \equiv 1\;(\hbox{mod }4)$ and $T$ has no space-like eigenvalue $1$ or $-1$.  
Then similarly as above, any choice of $S_o$ would have determinant necessarily $1$, and hence $T$ can not be real in $SO_o(n,1)$.

Suppose $T$ is parabolic. Then $T$ has a time-like non-degenerate indecomposable subspace 
$\V_1$ of dimension $3$, and $\V=\V_1 \oplus \W$, where $\W$ is the space-like $(n-2)$-dimensional orthogonal complement of $\V_1$. Note that $T|_{\V_1}$ has minimal polynomial $(x-1)^3$. Let $T_1=T|_{\V_1}$, $T_o=T|_{\W}$. Then $\V_1$ must be invariant under an isometry $S$ which conjugates $T$ to $T^{-1}$. Further $Q|_{\V_1}$ has signature $(2,1)$, hence we may consider $T_1$ as an unipotent isometry in $I(\H^2)$. 
It is easy to see that any isometry $S_1$ in $I(\H^2)$ which conjugates $T_1$ to  $T_1^{-1}$ must have determinant $-1$. Hence $T$ is real when there is an element $S_o$ in $O(n-2)$ such that $S_oT_oS_o^{-1}=T_o^{-1}$ and $\det S_o=-1$. This is the case precisely when $n-2 \nequiv 0\;(\hbox{mod }4)$ i.e. $n \nequiv 2\;(\hbox{mod }4)$, or $S_o$ has an eigenvalue $1$ or $-1$. It also follows that when $n \equiv 2\; (\hbox{mod }4)$ and $T$ has no space-like eigenvalue $1$ or $-1$, then any element $S$ which conjugates $T$ to $T^{-1}$ have determinant necessarily $-1$, and  hence $T$ can not be real in this case. 

This completes the proof of \thmref{realson1}. 

\subsubsection{Proof of \corref{1}}
\begin{proof}
Suppose $T$ is an isometry of $\H^{n+1}$. Suppose $T$ is real. It is enough to construct an involution $g$ in $SO_o(n,1)$ such that $gTg^{-1}=T^{-1}$. The construction of $g$  follows from the above theorem, and the decomposition of $\V$ described in the above proof.  
\end{proof}

\section{Classification of isometries}\label{class}
\subsection{The rotation angles}\label{ra}

\begin{definition}
 Let $S$ be an element in $SO(n)$. If the complex conjugate eigenvalues of $S$ are given by $\{e^{i\theta_j}, e^{-i\theta_j}\}$, $0 < \theta_j \leq \pi$, $j=1,...,k$, then $\{ \theta_1,...,\theta_k \}$ are called the \emph{rotation angles} of $S$. 

If an element $S$ in $SO(n)$ has $k$ rotation angles, then it is called a $k$-rotation.  
\end{definition}

\begin{remark}\label{krd}
If $S$ is a $k$-rotation, then $\E^n$ has a decomposition $D: \ \E^n=\oplus_{i=1}^k \V_i \oplus \Lambda^{n-2k}$, where $2k \leq n$, $S|_{\Lambda^{n-2k}}$ is the identity, each $\V_j$ has dimension $2$, and $S|_{\V_j}$ is a non-identity element of $O(2)$ with eigenvalues $\{e^{i \theta_j}, e^{-i \theta_j}\}$, $0<\theta_j\leq \pi$. The class $[D]$ in $\d(2; \ k)$ is called the \emph{decomposition} of $S$. 
\end{remark}

\begin{remark}
Classically the rotation angles are defined to be an element in $[-\pi, \pi]$. This works at dimension two. However, in higher dimensions, it is not possible to give a consistant definition of rotation angles in $[-\pi, \pi]$, since the choice of rotation angle at a two dimensional subspace in the above decomposition $D$ depends on a choice of orientation at this subspace. For example, 
take the standard block diagonal matrix $A$ in $SO(4)$ with two rotation angles
$\theta_1$ and $\theta_2$. Conjugating $A$ by the diagonal matrix ${diag}(1,-1,1,-1)$
gives the matrix $A^{-1}$ with rotation angles $-\theta_1$ and $-\theta_2$.  
This example shows that this notion of rotation angles depend on the choice of an orientation. 

To avoid this issue, to each rotation angle we assign a value in the interval $(0, \pi]$. This value depends on the complex conjugate eigenvalues of $T|_{\V_i}$ in the above decomposition. Since the eigenvalues are conjugacy invariants, hence the rotation angles are also conjugacy invariants in $O(n)$. 
\end{remark}

\begin{definition}
(i) Let $r>0$, and $T$ be a $k$-rotation of $\E^n$. The map $rT$ is called a \emph{$k$-rotatory stretch} of $\E^n$ with stretch factor $r$. 

(ii) Let $T$ be an orientation-preserving isometry of the Euclidean space $\E^n$ such that there is a decomposition into $T$-invariant subspaces $\Pi: \E^n=\oplus_{j=1}^k \V_j \oplus \Lambda^{n-2k}$, where for each $j=1,2,...,k$, $T|_{\V_j}$ is a non-identity element in $O(2)$ with eigenvalues $\{ e^{i \theta_j}, e^{-i \theta_j} \}$, and $T|_{\Lambda^{n-2k}}$ is a translation $x \mapsto x+a$. Such an element $T$ is called a \emph{$k$-rotatory translation} of $\E^n$. With respect to an orthonormal  coordinate system, 
the $k$-rotatory translations are of the form $Ax +b$, where $b$ is in $\E^n$ and $A$ is a $k$-rotation with the property that it has at least one eigenvalue $1$. 
\end{definition}

\begin{definition}
Let $T$ be an orientation-preserving conformal diffeomorphism of the sphere $\s^n$. Then $T$ is called a $k$-rotation of $\s^n$ if it has at least two fixed-points on $\s^n$, and it acts as a $k$-rotation on the complement of any of the fixed points. 
\end{definition}

\subsection{The classification of isometries of the hyperbolic space} 
In this section all isometries of $\H^{n+1}$ are assumed to be orientation-preserving. 

We consider the ball model of $\H^{n+1}$. Let $T$ be an elliptic isometry of $\H^{n+1}$, and $x$ be its fixed point in $\H^{n+1}$. Conjugating $T$ we can assume the fixed point to be $0$. Then $T$ is an element in $SO(n+1)$, and we can read the rotation angles. If $T$ has $k$ rotation angles, it is a $k$-rotatory elliptic. 

Suppose $T$ has more than one fixed point on $\H^{n+1}$. Then $T$ pointwise fixes a geodesic, and in particular $T$ fixes at least two points on the conformal boundary $\s^n$. When $k \leq [\frac{n}{2}]$, a $k$-rotatory elliptic has at least two fixed points on the hyperbolic space.  Thus, for $k \leq [\frac{n}{2}]$,  a $k$-rotatory elliptic restricts to a $k$-rotation of $\s^n$. Also a $k$-rotation of $\s^n$ extends uniquely to a $k$-rotatory elliptic by the Poincar\'e extension, cf. Ratcliffe \cite{rat}. 
In the upper-half space model of the hyperbolic space we identify $\s^n$ with the extended Euclidean space $\hat \E^n=\E^n \cup \{\infty\}$. In this case a $k$-rotation of $\s^n$ is conjugate to a $k$-rotation of $\E^n$. Hence, for $k \leq [\frac{n}{2}]$, a $k$-rotatory elliptic is conjugate to a $k$-rotation of $\E^n$. 

\medskip 
\emph{For the parabolic and hyperbolic cases we use the upper-half space model of $\H^{n+1}$. In these cases as above the conformal boundary $\s^n$ is identified with the extended Euclidean space. We shall follow this convention through out the paper. }

\medskip Let $T$ be a parabolic isometry. By conjugation assume that its unique fixed point on $\s^n$ is $\infty$. Then restriction of $T$ to $\E^n$ is a $k$-rotatory translation of the form $A_Tx+b$. In this case the associated orthogonal transformation $A_T$ must have an eigenvalue $1$. If $A_T$ has $k$ rotation angles, then $T$ is called a $k$-rotatory parabolic. 

If $T$ is hyperbolic, by conjugation, let its fixed points be $0$ and $\infty$. Then the restriction of $T$ to $\E^n$ is a $k$-rotatory stretch $rA_T$, $r \neq 1$. The rotation angles of $T$ are the rotation angles of $A_T$, and $T$ is called a $k$-rotatory hyperbolic. 

Gongopadhyay-Kulkarni \cite{kg} made  the above classification using the hyperboloid model of the hyperbolic space. In the hyperboloid model the isometry group is a linear algebraic group. The refined classification was obtained using the Jordan decomposition of an isometry. The notion of an \emph{angle} of an isometry in the hyperboloid model was earlier defined by Greenberg \cite{greenberg}. According to Greenberg an angle of an isometry was an  element in $[-\pi,0) \cup (0, \pi]$. In order to make them an invariant of the conjugacy class, we restrict them in the interval $(0, \pi]$. 

 It follows from the correspondence between the ball model and the hyperboloid model of the hyperbolic space that the number of rotation angles in the ball model is the same as the number of rotation angles in the hyperboloid model. Thus the refined classification of the isometries in the ball model is equivalent to the classification of Gongopadhyay-Kulkarni \cite{kg}.

\section{Topology of the conjugacy classes of regular $k$-rotations: Proof of \thmref{tcco}}\label{cco}
Let $\o_{k;\Theta} (n)$ denote the set of all $k$-rotations of $\E^n$ with rotation angles $\Theta=\{\theta_1, ..., \theta_k\}$, $0 < \theta_i \leq \pi$. Then $\o_{k;\Theta}$ is a conjugacy class in $O(n)$. 

Let $B(\theta)=\begin{pmatrix} \cos \theta & - \sin \theta
\\ \sin \theta & \cos \theta \end{pmatrix}$, and let  $\bar B(\theta)$ denote either of $B(\theta)$ or $B(-\theta)$. 

\subsection{The conjugacy classes of regular $k$-rotations of $\E^{2k}$: Proof of \thmref{tcco}.(i)} \ \ 

Let $T$ be a regular $k$-rotation of $\E^{2k}$. Then there is the eigenspace decomposition $D_T: \ \E^{2k}=\oplus_{i=1}^k {\V_i}$, where for each $i=1,...,k$, $\V_i$ is a two-dimensional  $T$-invariant subspace, and  \hbox{$T|_{\V_i}=\bar B(\theta_i)$}.  Thus with respect to $D_T$,  $T$ is of the form
$$\bar B(\theta_1,...,\theta_k)=\begin{pmatrix}\bar B(\theta_1) &  &  &  & \\ & \bar B(\theta_2) &  &  &  &  \\ & & \ddots & & \\ &  & & \bar B(\theta_k)  \end{pmatrix}.$$
Let $[D_T]$ denote the class of $D_T$ in $\d(2; \ k)$. Define the map $\alpha: \o_{k;\Theta}(2k) \to  \d(2; \ k)$ by \hbox{$\alpha( T)=[D_T]$}.
Consider an element  $[D]$ in $\d(2; \ k)$. Let $D: \E^{2k}=\oplus_{i=1}^k \W_i$ be a representative of this class.  Corresponding to the decomposition $[D]$, there are only  finitely many $k$-rotations in $\o_{k;\Theta}(2k)$. With respect to $D$ these $k$-rotations are of the form $\bar B(\Phi)$, where $\Phi$ is a permutation of 
$(\theta_1 , ..., \theta_k)$. For each $i = 1, 2, ..., k$, $\bar B(\theta_i )$ has two choices if $\theta_i \neq \pi$. If $\theta_i=\pi$ for some $i$, then $\bar B(\theta_i)$  has the unique choice $diag(-1, -1)$. Thus for each  permutation $\Phi$ there are
$c$ choices of k-rotations in $\o_{k,\Theta}(2k)$, where
$$c=\left \{ \begin{array}{ll} 2^{k} \ \hbox{ if  there is no rotation angle $\pi$, }\\
              2^{k-1} \ \hbox{ otherwise.}
             \end{array}\right.$$
There are $k!$ choices for $\Phi$. Hence the cardinality of $\alpha^{-1} ([D])$ is $d_0(k)=c k!$. 
For each $[D]$ in $\d(2; \  k)$, we identify $\alpha^{-1} ([D])$ with a finite set $F (\Theta)$ of cardinality $d_0(k)$. This gives us the following fibration with fiber $F (\Theta)$. 
$$\begin{CD}
F(\Theta) \\
@VVV \\
\o_{k;\Theta}(2k)\\
{\alpha}@VVV\\
 \d(2; \ k)
\end{CD}$$
With the topology induced by this fibration,  $\o_{k;\Theta}(2k)$ is a $d_0(k)$-sheeted covering space of $\d(2; \ k)$. 

\begin{remark}\label{2}
 For a non-regular $k$-rotations $\tau$, the decomposition of $\E^{2k}$ into the eigenspaces of $\tau$ is no more unique. For example, suppose $k=2$, and $\chi_{\tau}(x)=(x^2-2 \cos \theta \ x +1)^2$. Choose an eigenspace decomposition of $\tau$: $\V_1 \oplus \V_2$, and let $\{v_1, w_1, v_2, w_2\}$ be the corresponding orthonormal eigenbasis. Note that,  $\E^4=span(v_1, w_2) \oplus span(v_2, w_1)$ is also an eigenspace decomposition to $\tau$.  Hence it is not possible to assign to $\tau$ a unique decomposition as above. Thus the above construction for regular $k$-rotations does not hold  for non-regular $k$-rotations. Consequently, there is no canonical fibration as above for the conjugacy classes of non-regular $k$-rotations. 
\end{remark}

\subsection{The conjugacy classes of regular $k$-rotations of $\E^n$: Proof of \thmref{tcco}.(ii)}  \ \ 

Let $T$ be a regular $k$-rotation of $\E^n$. Then as in \remref{krd}, $T$ pointwise fixes an $(n-2k)$-dimensional subspace $\Lambda_T$ of $\E^n$. On the compliment $\Lambda^c_T$ of this subspace $T$ is a $k$-rotation. This defines a map
$\mu: \o_{k;\Theta}(n) \to \g_{n-2k}(n)$ given by $\mu(T)=\Lambda_T$. Clearly $\mu^{-1}(\Lambda_T)$ consists of all isometries which are $k$-rotations on $\Lambda^c_T \approx \E^{2k}$ with rotation angles $\Theta$. For each $\Lambda$ in $\g_{n-2k}(n)$ we identify $\mu^{-1}(\Lambda)$ with $\o_{k,\Theta}(2k)$. This induces a topology on $\o_{k;\Theta}(n)$ such that $\mu$ is a fibration with fiber $\o_{k,\Theta}(2k)$. 

This establishes \thmref{tcco}.

\section{Topology of the Conjugacy classes in $M_o(n)$: Proof of \thmref{tcc}}\label{cc}
When $n \equiv 2, \ 3 \ ( \hbox{mod }4)$, then it follows that two orientation-preserving isometries of $\H^{n+1}$ are conjugate in $M_o(n)$ if and only if they have the same classification based on their fixed points, have the same set of rotation angles, and their stretch factors are either the same or a unit.  When $n \equiv 1, \ 3 \ (\hbox{mod }4)$, then by \thmref{realson1} we see that it is true in many cases, except a few exceptional cases where an isometry $T$ is not conjugate to $T^{-1}$ in $M_o(n)$. However, it follows from Gongopadhyay-Kulkarni \cite[Theorem-1.2]{kg} that $T$ and $T^{-1}$ are conjugate in $M(n)$ for all values of $n$. Hence we extend the group $M_o(n)$ to the larger group $M(n)$ in these exceptional cases, and consider the conjugacy class of $T$ in $M(n)$.  With this convention the conjugacy classes of the (non-identity) orientation-preserving isometries of $\H^{n+1}$ are given by: 
$$\e_{k; \Theta}(n+1)=\{f \in M_o(n) \ | \  f \hbox{ is elliptic and has rotation angles } \Theta=\{\theta_1,...,\theta_k\} \ \},$$

 $\h_{k; r, \Theta}(n+1)=\{ f \in M_o(n) \ | \ f \hbox{ is hyperbolic with stretch factor $r$ or $r^{-1}$, and has rotation angles}$

 \hspace{2.5in} $\Theta=\{\theta_1,...,\theta_k\} \ \}, \ r \neq 1$
$$\p_{k;  \Theta}(n+1)=\{f \in M_o(n) \ | \ f \hbox{ is parabolic and has rotation angles } \Theta=\{\theta_1,...,\theta_k\} \ \},$$
where for each $i=1,2,...,k$, $0 < \theta_i \leq \pi$. 

In the following we shall implicitly use \cite[Theorem-1.1]{kg} while making assertions about the dynamics of the isometries. 

\subsection{Proof of \thmref{tcc}} Let $T$ be a regular isometry of $\H^{n+1}$ such that it has no rotation angle $\pi$.  

\subsubsection*{The space $\e_{k;\Theta}(n+1)$} Let $T$ be an element in  $\e_{k;  \Theta}(n+1)$, $k \leq [\frac{n}{2}]$. Then $T$ acts as a $k$-rotation on $\s^n$ and it has at least two fixed points on $\s^n$. In fact, $T$ pointwise fixes an $(n-2k)$-dimensional sphere on $\s^n$. Let $\s^{n-2k}_T \subset \s^n$ be the fixed-point sphere of $T$. Define the map $\varepsilon: \e_{k;  \Theta}(n+1) \to \mathcal S_{n-2k}(n): T \mapsto \s^{n-2k}_T$. Since the group $M_o(n)$ acts transitively on $\mathcal S_{n-2k}(n)$, by conjugation we identify an element of $\mathcal S_{n-2k}(n)$  with $\hat \E^{n-2k}$. With this identification, we have $\varepsilon^{-1}(\s^{n-2k}_T) \approx \o_{k;  \Theta}(2k) \approx \o_{k}(2k)$. This induces a topology on $\e_{k;  \Theta}(n+1)$ such that $\varepsilon$ is a fibration with fiber $\o_{k}(2k)$. 

\medskip Suppose $k=\frac{n+1}{2}$. This case occurs only when $n$ is odd.  We assert that every $\frac{n+1}{2}$-rotatory elliptic isometry of $\H^n$ has a unique fixed point on $\H^{n+1}$. To see this we use the ball model of the hyperbolic space.  If possible suppose an elliptic isometry $T$ has at least two fixed points $x$ and $y$ on the hyperbolic space. Since between two points there is a unique geodesic, $T$ must fixes the end points of the geodesic joining $x$ and $y$. Consequently $T$ fixes the geodesic pointwise, cf. Chen-Greenberg \cite[Lemma 3.3.2]{cg}. Thus the associated orthogonal transformation $A_T$ must have an eigenvalue $1$, and hence the number of rotation-angles is at most $\frac{n}{2}$. This is a contradiction. 

For such an isometry $T$, let $x_T$ denote its unique fixed point on $\H^n$. 
Let \hbox{$\e_{\frac{n+1}{2}; \Theta, x}(n+1)$} be the set of $\frac{n+1}{2}$-rotatory elliptics with rotation angles $\Theta$ and fixed-point $x$. By conjugation we can assume $x$ to be $0$ and identify $\e_{\frac{n+1}{2}; \Theta, x}(n+1)$ with $\e_{\frac{n+1}{2}; \Theta, 0}(n+1)$. Now the set $\e_{\frac{n+1}{2}; \Theta, 0}(n+1)$ is precisely the space of all $\frac{n+1}{2}$-rotations of $\E^{n+1}$. Hence we identify $\e_{\frac{n+1}{2}; \Theta, x}(n+1)$ with $\o_{\frac{n+1}{2};  \Theta}(n+1) \approx \o_{\frac{n+1}{2}}(n+1)$.

 Now define the map \hbox{$\psi: \e_{\frac{n+1}{2}; \Theta}(n+1) \to \H^{n+1}$} by $\psi(T)=x_T$. Clearly for $x$ in $\H^{n+1}$, 
$$\psi^{-1}(x)= \e_{\frac{n+1}{2}; \Theta, x}(n+1) \approx \o_{\frac{n+1}{2}}(n+1).$$ 
Hence with the induced topology, $\psi$ is a fibration with fiber $\o_{\frac{n+1}{2}}(n+1)$ and base $\H^{n+1}$. 

\subsubsection*{The space $\h_{k; r, \Theta}(n+1)$} Let $T$ be in $\h_{k;  r,  \Theta}(n+1)$.  Let $x_T$, $y_T$ be its fixed points on $\s^n$. Define the map $\varrho: \h_{k; r, \Theta}(n+1) \to \B_n: T \mapsto [x_T, y_T]$. Let $\h_{k; r,  x, y, \Theta}(n+1)$ denote the set of $k$-rotatory hyperbolic isometries with fixed points $x, \ y$, stretch factor $r$ or $r^{-1}$, and rotation angles $\Theta$.  
Since the group $M_o(n)$ acts doubly transitively on the conformal boundary, by conjugation we identify  
$\h_{k;  r,  x, y, \Theta}(n+1)$ and $\h_{k;  r,  0, \infty, \Theta}(n+1)$. An element in $\h_{k; r,  0, \infty, \Theta}(n+1)$ acts as a $k$-rotatory stretch of $\E^n$ with stretch factor $r$ or $r^{-1}$, and such a $k$-rotatory stretch with rotation angles $\Theta$ extends uniquely to an element in  $\h_{k;  r, 0, \infty, \Theta}(n+1)$.  Let $\r_{k; r, \Theta}(n)$ denote the set of $k$-rotatory stretches of $\E^n$ with stretch factor $r$ or $r^{-1}$, and  rotation angles $\Theta$. Then $\r_{k; r, \Theta}(n) \approx \o_{k; \Theta}(n) \times \{r, r^{-1}\} \approx \o_{k}(n)\times \{r , r^{-1} \}$ . Hence we identify $\h_{k;  r, 0, \infty, \Theta}(n+1)$ with $\o_{k}(n) \sqcup \o_k(n)$, where $\sqcup$ denote the disjoint union.  

So for $[x,y]$ in $\B_n$ we have $\varrho^{-1}([x,y])=\o_{k}(n)\sqcup \o_k(n)$. This induces a topology on $\h_{k;   \Theta}(n+1)$ such that  $\varrho$ is a fibration with fiber $\o_{k}(n) \sqcup \o_k(n)$. 

\subsubsection*{The space $\p_{k; \Theta}(n+1)$} Let $T$ be an element in $\p_{k;  \Theta}(n+1)$. Let $x_T$ be its unique fixed point on the conformal boundary $\s^n$. 
Let $\p_{k;  x, \Theta}$ denote the set of $k$-rotatory parabolics with fixed point $x$ and rotation angles $\Theta$. Since the group $M_o(n)$ acts transitively on the boundary, up to conjugacy we can assume the fixed point to be $\infty$, and for all $x$ on $\s^n$,  this identifies $\p_{k;  x, \Theta}(n+1)$ with $\p_{k;  \infty, \Theta}(n+1)$. 

Let $\t_{k; \Theta}(n)$ denote the set of $k$-rotatory translations of $\E^n$ with rotation angles $\Theta$. Then $\t_{k; \Theta}(n) \approx \o_{k;  \Theta}(n) \times {\R^n}^{\ast} \approx \o_{k}(n) \times {\R^n}^{\ast}$, where ${\R^n}^{\ast}=\R^n-\{{\bf 0}\} $.   Every element in $\p_{k; \infty,\Theta}(n+1)$ restricts to $\E^n$ as a $k$-rotatory translation and conversely by Poincar\'e extension every element of $\p_{k;  \infty,\Theta}(n+1)$ extends uniquely to an element in $\p_{k; \  \infty,\Theta}(n+1)$. This  identifies $\p_{k;  \infty,\Theta}(n+1)$ with $\t_{k; \Theta}(n)$. 

Define the map $\zeta: \p_{k;  \Theta}(n+1) \to \s^n$  by $\zeta(T)=x_T$. For a point $x$ in $\s^n$ the fiber is given by $\zeta^{-1} (x) \approx \o_{k}(n) \times {\R^n}^{\ast}$. This induces a topology on $\p_{k; \Theta}(n)$ such that $\zeta$ is a fibration with fiber $\o_{k}(n) \times {\R^n}^{\ast}$. 

\medskip It follows from the description of the conjugacy classes that they do not depend on the numerical values of the rotation angles, rather they only depend on the the number of rotation angles. Hence we identify
the conjugacy classes of regular k-rotatory elliptics, resp. parabolics, resp. hyperbolics with no rotation angle  $\pi$. 

This establishes \thmref{tcc}.

The case when one of the rotation angles is $\pi$ can be established similarly.

\subsubsection{Spaces of Hyperbolic isometries} 
Following the proof above, we have the description of the spaces of regular hyperbolic isometries with a given set of rotation angles. 

 Let $\h_{k;\Theta}(n+1)$ denote the set of regular $k$-rotatory hyperbolic isometries of $\H^{n+1}$ with rotation angles $\Theta$. Assume that $\pi$ is not a  rotation angle.  Let $\h_{k;  x, y, \Theta}(n+1)$ denote the set of $k$-rotatory hyperbolic isometries with fixed points $x, \ y$ and rotation angles $\Theta$.  
    Let $\r_{k; \Theta}(n)$ denote the set of $k$-rotatory stretches of $\E^n$ with rotation angles $\Theta$. Then $\h_{k; x,y,\Theta}(n+1) \approx \r_{k; \Theta}(n) \approx \o_{k; \Theta}(n) \times \R_+^{\ast} \approx \o_{k}(n)\times \R_+^{\ast}$ , where $\R^{\ast}_+$ denote the set of positive real numbers different from $1$. 

 Define the map $\varrho: \h_{k; \Theta}(n+1) \to \B_n: T \mapsto [x_T, y_T]$.  For $[x,y]$ in $\B_n$ we have $\varrho^{-1}([x,y])=\o_{k}(n)\times \R_+^{\ast}$. This induces a topology on $\h_{k;   \Theta}(n+1)$ such that  $\varrho$ is a fibration with fiber $\o_{k}(n) \times \R_+^{\ast}$.

\begin{acknowledgement}
 Wensheng Cao informed me about the work of Short et. al. \cite{sh1, sh2}. Ian Short told me about  his motivations \cite{bir, d} to call the real elements `reversible'.  I thank Cao and Short for their communications. Thanks are also due to Satya Deo for discussing the connectivity of $\tilde \Delta_n$.  
\end{acknowledgement}

\end{document}